\documentclass[12pt,leqno,final]{amsart}
\usepackage{amsmath,amsthm,amssymb,verbatim,times,enumerate,ifthen,graphics, graphicx}
\usepackage{caption,float}
\usepackage{pdfsync}
\usepackage[mathscr]{eucal}
\usepackage[utf8]{inputenc}
\usepackage[T1]{fontenc}
\usepackage{cite}
\usepackage[notref]{showkeys}
\usepackage[figurename={}]{caption}
\theoremstyle{plain}
\newtheorem{Thm}{Theorem}

\renewcommand{\theTHM}{}

\theoremstyle{remark}

\newcommand{\N}{\mathbb{N}}

\newcommand{\C}{\mathbb{C}}

\def\({\left(}
\def\){\right)}
\def\[{\left[}
\def\]{\right]}
\def\<{\left<}
\def\>{\right>}

\def\v{{\bf{v}}}
\def\w{{\bf{w}}}
\def\x{{\bf{x}}}
\def\y{{\bf{y}}}
\def\z{{\bf{z}}}
\def\a{{\bf{a}}}
\def\b{{\bf{b}}}
\def\c{{\bf{c}}}
\def\1{{\bf{1}}}

\newcommand{\eq}[1]{\eqref{E#1}}
\newcommand{\Eq}[2]{\ifthenelse{\equal{#1}{*}}
  {\begin{equation*}\begin{aligned}#2\end{aligned}\end{equation*}}
  {\begin{equation}\begin{aligned}\label{E#1}#2\end{aligned}\end{equation}}}

\begin{document}

\title{Products and inverses of  multidiagonal matrices with equally spaced diagonals}

\author{L\'aszl\'o Losonczi}
\address{Faculty of Economics, University of Debrecen, Hungary}
\email{laszlo.losonczi@econ.unideb.hu,losonczi08@gmail.com}

\date{\today}

\begin{abstract}
Let $n,k$ be fixed natural numbers with $1\le k\le n$ and let $A_{n+1,k,2k,\dots,sk}$ denote an $(n+1)\times (n+1)$ complex multidiagonal matrix having $s=[n/k]$ sub- and superdiagonals at distances $k,2k,\dots,sk$ from the main diagonal. We prove that the set $\mathcal{MD}_{n,k}$ of all such multidiagonal matrices is closed under multiplication and powers with positive exponents. Moreover the subset of $\mathcal{MD}_{n,k}$ consisting of all nonsingular matrices is closed under taking inverses and powers with negative exponents. In particular we obtain that the inverse of a nonsingular matrix
$A_{n+1,k}$ (called $k$-tridigonal) is in $\mathcal{MD}_{n,k}$, moreover if $n+1\le 2k$
then $A^{-1}_{n+1,k}$ is also $k$-tridigonal. Using this fact we give an explicite formula for this inverse.
\end{abstract}

\maketitle

\section{Introduction}

Multidiagonal matrices have a wide range of applications in various field of mathematics and engineering. Among them matrices with equally spaced diagonals have much nicer properties than those with arbitrarily spaced diagonals (see \cite{dFL2,BF2019,dFL} and their references). Here we study how  multidiagonal matrices with equally spaced diagonals behave under multiplication, taking inverse and powers.

Let $n,k$ be fixed natural numbers with $1\le k\le n$ and let $\mathcal{M}_{n}$
denote the set of  $n\times n$ complex matrices. A matrix $A=(a_{ij})_{i,j=0}^n,\in
\mathcal{M}_{n+1}$ is called $(k,2k,\dots,sk)$- multidiagonal if $a_{ij}=0$ if
$|i-j|\ne lk$, for $l=0,1,\dots,s$ where $sk\le n$. Such matrices will be denoted by
$A_{n+1,k,2k,\dots,sk}$ (supressing for the moments their dependence from the diagonals).
Such matrices are called $k$-tridiagonal if $s=1$ and $(k,2k)$-pentadiagonal if $s=2.$ Clearly the maximal number $s$ of sub- and superdiagonals in $A_{n+1,k,2k,\dots,sk}$ is $[n/k].$

Let $\mathcal{MD}_{n,k}$ be the set of all $A_{n+1,k,2k,\dots,sk}$ matrices with $s=[n/k].$ We prove that the set $\mathcal{MD}_{n,k}$ is closed under multiplication, and taking positive and (for nonsingular matrices also) negative powers. Since   matrices $A_{n+1,k,2k,\dots,s'k}$ with $1\le s'\le[n/k]$ also belong to $\mathcal{MD}_{n,k}$ (by taking the diagonals  $(s'+1)k,(s'+2)k,\dots,sk $ to be zero) we obtain that the inverse of a $k$-tridiagonal  matrix belongs to $\mathcal{MD}_{n,k}$. Moreover, if $n+1\le 2k$ then the inverse of a $k$-tridiagonal  matrix is also $k$-tridiagonal. Using this we find the explicit inverse of such $k$-tridiagonal  matrices.

The articles \cite{DL, MV} are related to the structure the product of tridiagonal matrices. Their investigations are based on the result that the product of two different $1$-tridiagonal Toeplitz matrices is  a $(1,2)$-pentadiagonal imperfect Toeplitz matrix (where the first and last elements of the main diagonal  are different to the other ones). In \cite{MA} the authors use Toeplitz $(1,2)$-pentadiagonal  matrices to study orthogonal polynomials on the unit circle.

\section{Multidiagonal matrices as the sum of their diagonals}

In the sequel (unless otherwise said) all matrices will be in
$\mathcal{M}_{n+1}$. Let now $A_{n+1,k,2k,\dots,sk}=(a_{ij})_{i,j=0}^n$  where  $s=[n/k]$. Denote its sub-, main, superdiagonal vectors extended to $n+1$ dimensional
vectors by adding the necessary number of zeros after their last coordinates by
 \Eq{vj}{
\v_{-s}&=(v_{-s,0},\dots,v_{-s,n-sk},0,\dots,0),\\
&\vdots \\
\v_{-1}&=(v_{-1,0},\dots,v_{-1,n-k},0,\dots,0),\\
\v_{0}&=(v_{0,0},\dots,v_{0,n}),\\
\v_{1}&=(v_{1,0},\dots,v_{1,n-k},0,\dots,0),\\
&\vdots \\
\v_{s}&=(v_{s,0},\dots,v_{s,n-sk},0,\dots,0).
}
This means that for $i,j=0,\dots,n$
\Eq{*}{
a_{ij}=\left\{\begin{array}{ll}
  v_{p,j}& \text{if}\; j-i=pk, p=-s,\dots,0,\dots,s,\\
  0 & \text{otherwise}.
  \end{array}
  \right.
}
For this matrix we also use the notations
\Eq{multi}{
A=A_{n+1,k,2k,\dots,sk}=A(\v_{-s},\dots,\v_0,\dots,\v_{s})
=A_{n+1,k,2k,\dots,sk}(\v_{-s},\dots,\v_0,\dots,\v_{s})
}
always trying to choose the most convenient one.  Here we have to remark that only the nonzero coordinates of the diagonal vectors  take part  in building our matrix.
Clearly all  matrices of $\mathcal{MD}_{n,k}$ can be written in the form \eq{multi}.

Introduce the elementary nilpotent matrix $N=(n_{ij})$ with
\Eq{*}{
n_{ij}=\left\{\begin{array}{rl}1&\hbox{ if } i-j=-1,\\
                          0&\hbox{ otherwise. }
\end{array}
\right.
}
$N$ contains one single unit superdiagonal right above the main diagonal, and its transpose $N^T$ contains one single unit subdiagonal immediately below the main diagonal.

It is easy to check that raising $N$ to power $k>0$ moves its single unit
superdiagonal to distance $k$ above the main diagonal.

The Moore-Penrose inverse $N^+$ of $N$ is its transpose, i.e $N^+=N^T$.
Let $N^0:=E$ (the unit matrix in $\mathcal{M}_{n+1}$) and  define the negative powers of $N$ by
\Eq{*}{
 N^{-k}:=(N^+)^k=(N^T)^k \quad(k\in \N).
 }
Then $N^{-k}$ has a single unit subdiagonal  at distance $k$  below the main diagonal. For $|k|\ge n+1$ the matrices $N^k$ become zero matrices.

Let
\Eq{*}{
D(\v):=\hbox{Diag}(v_0,v_1,\dots,v_n)
}
be the diagonal matrix with main diagonal $\v=(v_0,v_1,\dots,v_n)$.

The new form of our multidiagonal matrix is
\Eq{newform}{
&A(\v_{-s},\dots,\v_{0},\dots,\v_{s})=\\
&\quad N^{-sk}D(\v_{-s})+\dots+N^{-k}D(\v_{-1})+D(\v_{0})+D(\v_{1})N^{k}+\dots+D(\v_{s})N^{sk}.
}
 In this way we obtained our matrix as the sum of single diagonal matrices and it is
 easy to identify the matrix with the corresponding diagonal.
 We are grateful to Prof. Cs. Hegedűs for proposing us the use of the nilpotent matrix $N$ to describe multidiagonal matrices.

 Define the operator $\tau$ and its inverse  by
 \Eq{tau}{
 \tau \v:=(v_1,\dots,v_n,v_{n+1}),\quad \tau^{-1} \v:=(v_{-1},v_0,\dots,v_{n-1})
 }
 where for any vector $\v=(v_0,\dots,v_n)\in\C^{n+1}$
 \Eq{ext_v}{
 v_k=0  \hbox{ if }k>n \hbox{ or if }k<0.
 }
 This means that the effect of $\tau^k$ on any vector $\v$ is the increase of the subscripts of its coordinates by $k.$  Clearly  $\tau^j \v$ is zero vector for $|j|>n.$

 The $*$ product of two vectors $\v$ and ${\bf w}=(w_0,\dots,w_n)$ is defined
 coordinate-wise by
 \Eq{*}{
 \v*\w:=(v_0w_0,\dots,v_nw_n).
 }
 Clearly the operation $*$ is commutative, associative $D(\v)D(\w)=D(\v*\w)$ and
\Eq{prop-tau}{
\tau^i(\v*\w)&=(\tau^i\v)*(\tau^i\w), \,\,\tau^i(\tau^j\v)=\tau^{i+j}\v
}
for any $\v, \w\in\C^{n+1}$ and for integers  $i,j$.

For the multiplication of powers of $N$ we shall use the identities
\Eq{exp}{
N^iN^j&=N^{i+j}\,\,\hbox{ if }i,j\ge 0\,\,\hbox{or if }i,j\le 0, \\
}
and for nonnegative $i,j$
\Eq{powers}{
N^iN^{-j}&=D(\tau^i\1)N^{i-j}&\quad\hbox{if }i-j\ge 0,\\
N^{-j}N^i&=D(\tau^{-j}\1)N^{i-j}&\quad\hbox{if }i-j\ge 0,\\
N^iN^{-j}&=N^{-(j-i)}D(\tau^j\1)&\quad\hbox{if }i-j\le 0,\\
N^{-j}N^i&=N^{-(j-i)}D(\tau^{-i}\1)&\quad\hbox{if }i-j\le 0,\\
}
where $\1=(1,\dots,1)\in\C^{n+1}$ is the unit vector. Please note that $\tau^i\1*\v=\tau^i\v$ for integer $i.$

The order of factors in the products  $D(\v)N^{\pm j}$ can be changed by help of the identities
\Eq{comm}{
&D(\v)N^{-j}=N^{-j}D(\tau^j\v),\quad\,\,\, N^{j}D(\v)=D(\tau^j\v)N^{j},\\
&N^{-j}D(\v)=D(\tau^{-j}\v)N^{-j},\quad D(\v)N^{j}=N^{j}D(\tau^{-j}\v),
}
valid for any $\v\in \C^{n+1}$ and for nonnegative integer values of $j$.

\section{Structure of products inverses and powers of some multidiagonal matrices}

 \begin{Thm}(i) The set $\mathcal{MD}_{n,k}$ is closed under multiplication and taking powers with positive exponents.
\newline\noindent (ii) The subset of $\mathcal{MD}_{n,k}$ consisting of all nonsingular matrices is closed under taking inverses and powers with negative (and also nonnegative) exponents.
\end{Thm}
\begin{proof}
Let
\Eq{V&W}{
V&=\sum_{i=1}^{s}N^{-ik}D(\v_{-i})+\sum_{i=0}^{s}D(\v_{i})N^{ik}\\
W&=\sum_{j=1}^{s}N^{-jk}D(\w_{-j})+\sum_{j=0}^{s}D(\w_{j})N^{jk}
}
be two matrices in $\mathcal{MD}_{n,k}$ where the vectors $\v_i,(i=-s,\dots,0,\dots,s)$ are defined by \eq{vj}, and
 \Eq{*}{
\w_{j}=(w_{j,0},\dots,w_{j,n-jk},0,\dots,0)\quad (j=-s,\dots,0,\dots,s).
}
The product $VW$ is decomposed into four sums
\Eq{VW}{
VW&=\sum_{i=1}^s\sum_{j=1}^sN^{-ik}D(\v_{-i})N^{-jk}D(\w_{-j})
  +\sum_{i=1}^s\sum_{j=0}^sN^{-ik}D(\v_{-i})D(\w_{j})N^{jk}\\
  &+\sum_{i=0}^s\sum_{j=1}^sD(\v_{i})N^{ik}N^{-jk}D(\w_{-j})
  +\sum_{i=0}^s\sum_{j=0}^sD(\v_{i})N^{ik}D(\w_{j})N^{jk}.
}
We transform the summands  by help of the relations \eq{powers} and \eq{comm} as follows:
\Eq{*}{
\begin{array}{rl}
&N^{-ik}D(\v_{-i})N^{-jk}D(\w_{-j})=N^{-ik}N^{-jk}D(\tau^{ik}\v_{-i})D(\w_{-j})\\
&\quad =N^{-(i+j)k}D(\tau^{ik}\v_{-i}*\w_{-j}),\\
&N^{-ik}D(\v_{-i})D(\w_{j})N^{jk}=N^{-ik}D(\v_{-i}*\w_{j})N^{jk}
=D(\tau^{-ik}(\v_{-i}*\w_{j}))N^{-ik}N^{jk},\\
&D(\v_{i})N^{ik}N^{-jk}D(\w_{-j})=N^{ik}D(\tau^{-ik}\v_{i})N^{-jk}D(\w_{-j})\\
&\quad =N^{ik}N^{-jk}D(\tau^{(j-i)k}\v_{i}*\w_{-j}),\\
&D(\v_{i})N^{ik}D(\w_{j})N^{jk}=D(\v_{i}*\tau^{ik}\w_{j})N^{(i+j)k}.
\end{array}
}

The expressions in the second and third line require further transformations using again \eq{powers},\eq{comm} and the properties of the $*$. The expression in the second line is transformed as follows.

If $i-j\le 0$ then we get
\Eq{*}{
&D(\tau^{-ik}(\v_{-i}*\w_{j})N^{-ik}N^{jk}
=D(\tau^{-ik}(\v_{-i}*\w_{j}))D(\tau^{-ik}\1)N^{(j-i)k}\\
&\quad =D(\tau^{-ik}(\v_{-i}*\w_{j}))N^{(j-i)k},
}
since
\Eq{*}{
\tau^{-ik}(\v_{-i}*\w_{j})*\tau^{-ik}\1=\tau^{-ik}(\v_{-i}*\w_{j}*\1)
=\tau^{-ik}(\v_{-i}*\w_{j}).
}
If $i-j>0$ then  we obtain
\Eq{*}{
&D(\tau^{-ik}(\v_{-i}*\w_{j}))N^{-ik}N^{jk}
=D(\tau^{-ik}(\v_{-i}*\w_{j}))D(\tau^{-jk}\1)N^{-(i-j)k}\\
&\quad=D\(\tau^{-ik}(\v_{-i}*\w_{j})*\tau^{-jk}\1\)N^{-(i-j)k}\\
&\qquad=N^{-(i-j)k}D\(\tau^{(i-j)k}(\tau^{-ik}(\v_{-i}*\w_{j})*\tau^{-jk}\1)\)
=N^{-(i-j)k}D(\tau^{-jk}(\v_{-i}*\w_{j})),
}
since
\Eq{*}{
&\tau^{(i-j)k}\(\tau^{-ik}(\v_{-i}*\w_{j})*\tau^{-jk}\1\)
=\tau^{(i-j)k}(\tau^{-ik}\v_{-i}*\tau^{-ik}\w_{j}*\tau^{-jk}\1)\\
&\quad=\tau^{(i-j)k}\(\tau^{-ik}\v_{-i}*\tau^{-jk}(\tau^{(j-i)k}\w_{j}*\1)\)
=\tau^{(i-j)k}(\tau^{-ik}\v_{-i}*\tau^{-ik}\w_{j})\\
&\qquad=\tau^{-jk}(\v_{-i}*\w_{j}).
}
We transform the expression in the third line similarly.

If $i-j\le 0$ then we get
\Eq{*}{
&N^{ik}N^{-jk}D(\tau^{(j-i)k}\v_{i}*\w_{-j})
=N^{-(j-i)k}D(\tau^{jk}\1)D(\tau^{(j-i)k}\v_{i}*\w_{-j})\\
&\quad=N^{-(j-i)k}D(\tau^{jk}\1*\tau^{(j-i)k}\v_{i}*\w_{-j})
=N^{-(j-i)k}D(\tau^{jk}(\1*\tau^{-ik}\v_{i})*\w_{-j})\\
&\qquad=N^{-(j-i)k}D(\tau^{(j-i)k}\v_{i}*\w_{-j}).
}
For $i-j>0$ we have
\Eq{*}{
&N^{ik}N^{-jk}D(\tau^{(j-i)k}\v_{i}*\w_{-j})
=D(\tau^{ik}\1)N^{(i-j)k}D(\tau^{(j-i)k}\v_{i}*\w_{-j})\\
&\quad=D(\tau^{ik}\1)D\(\tau^{(i-j)k}(\tau^{(j-i)k}\v_{i}*\w_{-j})\)N^{(i-j)k}
=D(\v_{i}*\tau^{(i-j)k}\w_{-j})N^{(i-j)k}
}
since
\Eq{*}{
&\tau^{ik}\1*\tau^{(i-j)k}(\tau^{(j-i)k}\v_{i}*\w_{-j})
=\tau^{ik}\1*\v_{i}*\tau^{(i-j)k}\w_{-j}\\
&\quad =\tau^{ik}(\1*\tau^{-ik}\v_{i})*\tau^{(i-j)k}\w_{-j}=\v_{i}*\tau^{(i-j)k}\w_{-j}.
}
Using these new forms of the summands and splitting the second and third sums into two we can rewrite \eq{VW} as
\Eq{VWnew}{
VW&\!=\!\sum_{i=1}^s\sum_{j=1}^sN^{-(i+j)k}D(\tau^{ik}\v_{-i}*\w_{-j})
  +\sum_{i=1}^s\sum_{j=0,i\le j}^sD(\tau^{-ik}(\v_{-i}*\w_{j}))N^{(j-i)k}\\
  &+\sum_{i=1}^s\!\sum_{j=0,i\!>\!j}^sN^{-\!(i\!-\!j)k}\!D(\tau^{-\!jk}(\v_{-\!i}\!*\!\w_{j}))
  \!+\!\sum_{i=0}^s\!\sum_{j=1,i\!\le\! j}^sN^{-\!(j\!-\!i)k}\!D(\tau^{(j\!-\!i)k}\v_{i}\!*\!\w_{-\!j})\\
  &+\sum_{i=0}^s\sum_{j=1,i>j}^sD(\v_{i}*\tau^{(i-j)k}\w_{-j})N^{(i-j)k}
  +\sum_{i=0}^s\sum_{j=0}^sD(\v_{i}*\tau^{ik}\w_{j})N^{(i+j)k}.
}

Using the rules $D(\v)N^p+D(\w)N^p=D(\v+\w)N^p,$ $N^{-p}D(\v)+N^{-p}D(\w)=N^{-p}D(\v+\w)$
for $p\ge 0, \v,\w\in \C^{n+1}$ we add those terms of \eq{VWnew} for which the exponents of $N$ are the same nonnegative or negative numbers and omit those terms where the absolute value of the  exponents of $N$ is greater than $n$.

The result is
\Eq{VWz}{
VW=\sum_{p=1}^{s}N^{-pk}D(\z_{-p})+D(\z_{0})+\sum_{p=1}^{s}D(\z_{p})N^{pk}
}
with suitable vectors $\z_p\,(p=-s,\dots,0,\dots,s)$ proving that the set $\mathcal{MD}_{n,k}$ is closed under taking products. This clearly implies that it is also closed under taking powers with positive exponents,  completing the proof of (i).

To prove (ii) take a nonsingular matrix $V\in \mathcal{MD}_{n,k}$ and let
\Eq{*}{
\hbox{Det}(V-\lambda E)=\sum\limits_{j=0}^{n+1} \nu_j\lambda^j
}
be the characteristic polynomial of $V$, where $\nu_j\in \mathbb{C},$ in particular
$\nu_{n+1}=(-1)^{n+1}$ and $\nu_0=\hbox{Det}(V)\ne 0$. By the Cayley-Hamilton theorem we have
$\sum\limits_{j=0}^{n+1} \nu_jV^j=O$ (where $O$ is the zero matrix) therefore
\Eq{*}{
E=V\(-\sum\limits_{j=1}^{n+1} \frac{\nu_j}{\nu_0}V^{j-1}\)
=\(-\sum\limits_{j=1}^{n+1} \frac{\nu_j}{\nu_0}V^{j-1}\)V
}
showing that
\Eq{*}{
V^{-1}=-\sum\limits_{j=1}^{n+1} \frac{\nu_j}{\nu_0}V^{j-1}\in \mathcal{MD}_{n,k}
}
and completing the proof.
\end{proof}

\section{Explicite form of the inverse of the $k$-tridiagonal matrix $A_{n+1,k}$  if $n+1\le 2k$}

\begin{Thm} (j) If $n+1\le 2k$ then the $k$-tridiagonal matrix
\Eq{inverse}{
A=N^{-k}D(\a)+D(\b)+D(\c)N^{k}
}
where
\Eq{*}{
\a=(a_{0},\dots,a_{n-k},0,\dots,0),\b=(b_{0},\dots,b_{n}),\c=(c_{0},\dots,c_{n-k},0,\dots,0)
}
is nonsingular if and only if 
\Eq{nonsing}{
b_j&\ne 0\,(j=n+1-k,\dots,k-1)\\
  b_jb_{j+k}-a_{j}c_{j}&\ne 0,\,\,(j=0,\dots,n-k)).
}
(jj) If \eq{nonsing} holds then  $A^{-1}$ is also  $k$-tridiagonal and is of the form
\Eq{*}{
A^{-1}&\!=\!N^{-k}D(\x)+D(\y)+D(\z)N^{k}
}
where
\Eq{*}{
\x&\!=\!\(\frac{-a_{0}}{b_{0}b_{k}\!-\!a_{0}c_{0}},\dots,
\frac{-a_{n\!-\!k}}{b_{n\!-\!k}b_{n}\!-\!a_{n-k}c_{n\!-\!k}},0,\dots,0\),\\\\
 \y&\!=\!\(\frac{b_{k}}{b_{0}b_{k}\!-\!a_{0}c_{0}},\dots,
\frac{b_{n}}{b_{n\!-\!k}b_{n}\!-\!a_{n\!-\!k}c_{n\!-\!k}},
\frac{1}{b_{n\!+\!1\!-\!k}},\dots,\frac{1}{b_{k\!-\!1}},
\frac{b_{0}}{b_{k}b_{0}\!-\!a_{0}c_{0}},\dots,
\frac{b_{n\!-\!k}}{b_{n}b_{n\!-\!k}\!-\!a_{n\!-\!k}c_{n\!-\!k}}\),\\\\
 \z&\!=\!\(\frac{-c_{0}}{b_{0}b_{k}\!-\!a_{0}c_{0}},\dots,
\frac{-c_{n\!-\!k}}{b_{n\!-\!k}b_{n}\!-\!a_{n\!-\!k}c_{n\!-\!k}},0,\dots,0\).
}
\end{Thm}
\begin{proof}
 The determinant of $A$ is by the known formula (see e.g. \cite{EA})
\Eq{Adet}{
\hbox{Det}\,A=\prod\limits_{j=0}^n f_j
}
where
\Eq{*}{ 
f_{j}=\left\{\begin{array}{rll}
&b_j &\hbox{ if } j=0,\dots,k-1,\\
&b_j-a_{j-k}c_{j-k}/f_{j-k}&\hbox{ if } j=k,\dots,n.
\end{array}\right.
}
To define $f_j$ for $j=k,\dots,n$ we have to assume  $f_j\ne 0$ for $j=0,\dots,n-k.$
However formula \eq{Adet} is valid without this assumption as after simplifications
the fractions disappear (see \cite{dFL}).
In our case   $n-k\le k-1$ and the product in \eq{Adet} can be simplified to
\Eq{*}{  
\hbox{Det}\,A&=\(\prod\limits_{j=0}^{k-1} b_j\)\(\prod\limits_{j=k}^{n}
\(b_j\!-\!a_{j\!-\!k}c_{j\!-\!k}/f_{j-k}\)\)
\!=\!\(\prod\limits_{j=0}^{k-1} b_j\)\(\prod\limits_{j=0}^{n-k}
\(b_{j+k}\!-\!a_{j}b_{j}/b_{j}\)\)\\
&=\!\(\prod\limits_{j=n+1-k}^{k-1}\!\! b_j\)\(\prod\limits_{j=0}^{n-k}
\(b_jb_{j\!+\!k}\!-\!a_{j}c_{j}\)\).
}
This shows that $A$ is nonsingular if and only if \eq{nonsing} holds, proving (j).
\medskip

If $n+1\le 2k$ and \eq{nonsing} holds then we have seen that $A^{-1}$ is also $k$-tridiagonal thus we may write it as
\Eq{*}{ 
A^{-1}=X=N^{-k}D(\x)+D(\y)+D(\z)N^{k}
}
where 
 \Eq{*}{
\x=(x_{0},\dots,x_{n-k},0,\dots,0),\,\y=(y_{0},\dots,y_{n}),\,
\z=(z_{0},\dots,z_{n-k},0,\dots,0).
}
Expanding the product $AX$ we get
\Eq{expand}{
AX&=N^{-k}D(\a)N^{-k}D(\x)+N^{-k}D(\a)D(\y)+N^{-k}D(\a)D(\z)N^{k}\\
&+D(\b)N^{-k}D(\x)+D(\b)D(\y)+D(\b)D(\z)N^{k}\\
&+D(\c)N^{k}N^{-k}D(\x)+D(\c)N^{k}D(\y)+D(\c)N^{k}D(\z)N^{k}.
}
Using  suitable relations of \eq{comm} we rewrite the first term of \eq{expand} as
\Eq{*}{
N^{-k}D(\a)N^{-k}D(\x)=N^{-k}N^{-k}D(\tau^k\a)D(\x)=N^{-2k}D(\tau^k\a*\x),
}
the  second term as $N^{-k}D(\a*\y).$

The third term can be written as
\Eq{*}{
N^{-k}D(\a)D(\z)N^{k}&=D(\tau^{-k}\a)N^{-k}D(\z)N^{k}
=D(\tau^{-k}\a)D(\tau^{-k}\z)N^{-k}N^{k}\\
&=D(\tau^{-k}(\a*\z)D(\tau^{-k}\1)N^{0}=D(\tau^{-k}(\a*\z)).
}
Rewriting the other terms in a similar way we finally get that
\Eq{AX}{
AX&=N^{-2k}D(\tau^k\a*\x)+N^{-k}D(\a*\y+\tau^k\b*\x)\\
&+D(\tau^{-k}(\a*\z)+\b*\y+\c*\x)\\
&+D(\b*\z+\c*\tau^k\y)N^{k}+D(\c*\tau^k\z)N^{2k}.
}
In our case $N^{\pm 2k}$= zero matrix, hence the equations of the linear inhomogeneous system $AX=E$  can be written as
\Eq{system}{
&\a*\y+\tau^k\b*\x={\bf 0}, \,\b*\z+\c*\tau^k\y={\bf 0},\\
&\tau^{-k}(\a*\z)+\b*\y+\c*\x=\1
}
where ${\bf 0}$ is the $n+1$ dimensional zero vector. The unknowns are the nonzero coordinates of $\x,\y,\z$ numbering to $n+1+2(n+1-k)=3(n+1)-2k.$  System \eq{system} is in detailed form
\Eq{*}{ 
&{\bf 0}=\a*\y+\tau^k\b*\x
=(a_0,\dots,a_{n-k},0,\dots,0)*(y_{0},\dots,y_{n})\\
&+(b_{k},\dots,b_{n},0,\dots,0)*(x_{0},\dots,x_{n-k},0,\dots,0)\\
&=(\underbrace{b_{k}x_{0}+a_{0}y_{0},\dots,b_{n}x_{n-k}+a_{n-k}y_{n-k}}_{n+1-k},
\underbrace{0,\dots,0}_k))\\\\
&{\bf 0}=\b*\z+\c*\tau^k\y
=(b_{0},\dots,b_{n})*(z_{0},\dots,z_{n-k},0,\dots,0\\
&+(c_{0},\dots,c_{n-k},0,\dots,0)*(y_{k},\dots,y_{n},0,\dots,0)\\
&=(\underbrace{c_{0}y_{k}+b_{0}z_{0},\dots,c_{n-k}y_{n}+b_{n-k}z_{n-k}}_{n+1-k},
\underbrace{0,\dots,0}_k)\\\\
&\1=\tau^{-k}(\a*\z)+\b*\y+\c*\x=(0,\dots,0,a_{0}z_{0},\dots,a_{n-k}z_{n-k})\\
&+(b_{0},\dots,b_{n})*(y_{0},\dots,y_{n})
+(c_{0},\dots,c_{n-k},0,\dots,0)*(x_{0},\dots,x_{n-k},0,\dots,0)\\
&=(\underbrace{c_{0}x_{0}+b_{0}y_{0},\dots,c_{n-k}x_{n-k}+b_{n-k}y_{n-k}}_{n+1-k},
\underbrace{0,\dots,0}_k)\\
&+(\underbrace{0,\dots,0}_{n+1-k},\underbrace{b_{n+1-k}y_{n+1-k},\dots,b_{k-1}y_{k-1}}_{2k-(n+1)},
\underbrace{0,\dots,0}_{n+1-k})\\
&+(\underbrace{0,\dots,0}_k,
\underbrace{b_{k}y_{k}+a_{0}z_{0},\dots,b_{n}y_{n}+a_{n-k}z_{n-k}}_{n+1-k}).
}
In the first and second group the last $k$ equations are trivial ($0=0$) thus these
are  omitted. The remaining number of our (non trivial) equations is $2(n+1-k)+n+1=3(n+1)-2k$, the same as the number of unknowns.

Next we solve this system. The unknowns $y_{n+1-k},\dots,y_{k-1}$ obtained easily as
\Eq{*}{
y_j=\frac{1}{b_j}\,\,(j=n+1-k,\dots,k-1).
}
Collect the remaining unknowns into one column vector and the corresponding free
terms also into one vector
\Eq{*}{
\x^*&=(x_{0},\dots,x_{n-k},y_{0},\dots,y_{n-k},y_{k},\dots,y_{n},z_{0},\dots,z_{n-k})^T\\
\b^*&=(\underbrace{0,\dots,0}_{2(n+1-k)},\underbrace{1,\dots,1}_{2(n+1-k)})^T.
}
Denoting by $U$ the matrix of the reduced  system it can be written as $U\x^*=\b^*.$

This reduced system has $4(n+1-k)$ equations and unknowns. In detailed form
\Eq{*}{ 
\renewcommand\arraystretch{0.2}\arraycolsep=1.0pt
 \left(\!
  \begin{array}{ccc|ccc|ccc|ccc}
   b_{k}x_{0}  &  && a_{0}y_{0}   &  &  &  & &   &  &  &  \\
    &\ddots   &  && \ddots   &  &  & & &  &  &    \\
     & &b_{n}x_{n\!-\!k}   &&  & a_{n\!-\!k}y_{n\!-\!k}    &  &  &  &  &  &  \\
    \hline
     &  &  &  &  &    & c_{0}y_{k} &  &  & b_{0}z_{0} &  &  \\
     &  &  &  &  &    &  & \ddots &  &  & \ddots &  \\
     &  &  &  &  &    &  &  & c_{n\!-\!k}y_{n} &  &  & b_{n\!-\!k}z_{n\!-\!k} \\
    \hline
    c_{0}x_{0} &  &  & b_{0}y_{0} &  &  & & &  &  &  &    \\
     & \ddots  &  &  &\ddots  &  &  & & &  &  &    \\
     &  & c_{n\!-\!k}x_{n\!-\!k} &  &  &b_{n\!-\!k}y_{n\!-\!k}  &  & & &  &  &    \\
    \hline
       &  &  &  & & &b_{k}y_{k}  &  &  & a_{0}z_{0} &  &  \\
       &  &  &  & & &  &\ddots  &  &  & \ddots &  \\
       &  &  &  & & &  &  &b_{n}y_{n}  &  &  &a_{n\!-\!k}z_{n\!-\!k}  \\
  \end{array}
\!\right)
\!=\!
\renewcommand\arraystretch{0.52}\arraycolsep=1.60pt
\left(\!
   \begin{array}{c}
     0 \\
     \vdots \\
     0 \\
     \hline
     0 \\
     \vdots \\
     0 \\
     \hline
     1 \\
     \vdots \\
     1 \\
     \hline
     1 \\
      \vdots \\
     1 \\
   \end{array}
 \!\right)
}
which shows that our system consists of four groups of equations, each of them with
$n+1-k$ equations of similar structures.
Number the equations starting by zero. Multiply the  $j$th equations of the first
system by $-c_{j}$ and add these to the $j$th equations of the third system
multiplied by $b_{k+j}$  for $j=0,\dots,n-k$.
Our system goes over into
\Eq{*}{  
\renewcommand\arraystretch{0.2}\arraycolsep=1.0pt
 \left(\!
  \begin{array}{ccc|ccc|ccc|ccc}
   b_{k}x_{0}  &  && a_{0}y_{0}   &  &  &  & &   &  &  &  \\
    &\ddots   &  && \ddots   &  &  & & &  &  &    \\
     & &b_{n}x_{n\!-\!k}   &&  & a_{n\!-\!k}y_{n\!-\!k}    &  &  &  &  &  &  \\
    \hline
     &  &  &  &  &    & c_{0}y_{k} &  &  & b_{0}z_{0} &  &  \\
     &  &  &  &  &    &  & \ddots &  &  & \ddots &  \\
     &  &  &  &  &    &  &  & c_{n\!-\!k}y_{n} &  &  & b_{n\!-\!k}z_{n\!-\!k} \\
    \hline
     &  &  & (b_{0}b_{k}\!-\!a_{0}c_{0})y_{0} &  &  & & &  &  &  &    \\
     &   &  &  &\ddots  &  &  & & &  &  &    \\
     &  &  &  &  &(b_{n\!-\!k}b_{n}\!-\!a_{n\!-\!k}c_{n\!-\!k})y_{n\!-\!k}  &  & & &
     &  &    \\
    \hline
       &  &  &  & & &b_{k}y_{k}  &  &  & a_{0}z_{0} &  &  \\
       &  &  &  & & &  &\ddots  &  &  & \ddots &  \\
       &  &  &  & & &  &  &b_{n}y_{n}  &  &  &a_{n\!-\!k}z_{n\!-\!k}  \\
  \end{array}
\!\right)
\!=\!
\renewcommand\arraystretch{0.52}\arraycolsep=1.60pt
\left(\!
   \begin{array}{c}
     0 \\
     \vdots \\
     0 \\
     \hline
     0 \\
     \vdots \\
     0 \\
     \hline
     b_{k} \\
     \vdots \\
      b_{n} \\
     \hline
     1 \\
      \vdots \\
     1 \\
   \end{array}
 \!\right)
}

From the third group of  equations  we get immediately that
\Eq{*}{  
y_j=\frac{b_{k+j}}{b_{j}b_{k+j}-a_{j}c_{j}}\,\,(j=0,\dots,n-k).
}
To continue our calculations we temporally assume that $b_{k+j}\ne 0,\,\,(j=0,\dots,
n-k).$ Then from the first group of equations we obtain that
\Eq{x}{
x_j=\frac{-a_{j}b_{k+j}}{b_{k+j}}=
\frac{-a_{j}}{b_{j}b_{k+j}-a_{j}c_{j}}\,\,(j=0,\dots,n-k).
}
Multiply the $j$th equations of the second group by $-a_j$ and add them to the $j$th
equations of the fourth group multiplied by  $b_j$ for $j=0,\dots,n-k$. Then the
fourth group of equations go over into
\Eq{*}{
(b_{k+j}b_{j}-a_{j}c_{j})y_{k+j}=b_{j},
}
hence
\Eq{*}{  
y_{k+j}=\frac{b_{j}}{b_{k+j}b_{j}-a_{j}c_{j}},\,\,(j=0,\dots,n-k).
}
Finally multiply the $j$th equations of the second group by $-b_{j+k}$ and add them
to the $j$th equations of the fourth group multiplied by  $c_j$ for $j=0,\dots,n-k$.
Then the fourth group of equations become
\Eq{*}{
(-b_{k+j}b_{j}+a_{j}c_{j})z_{j}=c_{j},
}
thus
\Eq{*}{  
z_{j}=\frac{-c_{j}}{b_{k+j}b_{j}-a_{j}c_{j}},\,\,(j=0,\dots,n-k).
}
Now we justify  \eq{x} without our temporally assumption. Namely if  $b_{k+j}=0$ for
some $j=0,\dots, n-k$ then change it a little to $b'_{k+j}\ne 0$ such  that the
factor $b_jb'_{j+k}-a_{j}c_{j}\ne 0$. Then we obtain
\Eq{*}{
x'_j=\frac{-a_{j}}{b_{j}b'_{k+j}-a_{j}c_{j}}
}
taking the limit $b'_{k+j}\to 0=b_{k+j}$ justifies  the validity of the final formula
for $x_j$. 
\end{proof}

\end{document}